\documentclass[12pt]{amsart}

\usepackage[T1]{fontenc}
\usepackage[utf8]{inputenc}
\usepackage{amsmath,amssymb,amsthm,mathtools}
\usepackage{tikz-cd}
\usepackage[margin=1.15in]{geometry}
\usepackage{microtype}
\usepackage{needspace}
\usepackage{xcolor}
\definecolor{linkblue}{HTML}{154C79}
\usepackage[colorlinks=true,allcolors=linkblue]{hyperref}
\hypersetup{pdftitle={Rational badness of wedges},
  pdfauthor={Nikita Golub},
  pdfsubject={Rational completion and the homology of wedges}}

\newtheorem{theorem}{Theorem}[section]
\newtheorem{proposition}[theorem]{Proposition}
\newtheorem{lemma}[theorem]{Lemma}
\newtheorem{corollary}[theorem]{Corollary}
\theoremstyle{remark}

\newtheorem{example}[theorem]{Example}
\numberwithin{equation}{section}

\newcommand{\Q}{\mathbb Q}
\newcommand{\Z}{\mathbb Z}
\newcommand{\LL}{\mathbb L}
\newcommand{\F}{\mathbb F}
\newcommand{\K}{\mathbb K}
\newcommand{\cont}{\mathfrak c}
\newcommand{\CE}{\mathcal C}
\newcommand{\real}[1]{\langle #1\rangle}
\newcommand{\comp}[1]{#1^{\wedge}_{\Q}}
\DeclareMathOperator{\ad}{ad}
\DeclareMathOperator{\MC}{MC}
\DeclareMathOperator{\Tot}{Tot}
\DeclareMathOperator{\Sym}{Sym}
\DeclareMathOperator{\coker}{coker}
\DeclareMathOperator{\Frac}{Frac}

\title{Wedges, Rational completion and \(\Q\)-bad spaces}
\author{Nikita Golub\textsuperscript{\(\dagger\)}}
\thanks{\textsuperscript{\(\dagger\)} Mathematics Program and Center for Quantum and Topological Systems
(CQTS), New York University Abu Dhabi, UAE. Email: \texttt{ngxy@nyu.edu}. Department of Mathematics and Computer Science, St. Petersburg State University, St. Peterburg, RU}
\thanks{\textit{Acknowledgements.} The author thanks Sergei Ivanov for suggesting the problem, for helpful
discussions, and for his hospitality during the author's stay at BIMSA
in China. The author also gratefully acknowledges financial support
from a scholarship in honour of V.~A.~Rokhlin.}
\date{September 11, 2026}

\begin{document}

\begin{abstract}
Let $A$ and $X$ be connected based CW complexes with $H_1(A;\Q)\neq0$,
and let $m\geq1$ be the least degree with $\widetilde H_m(X;\Q)\neq0$.
We prove that the Bousfield--Kan rational completion map of $A\vee X$
induces a map on $H_{2m}(-;\Q)$ whose cokernel has dimension at least
$2^{\aleph_0}$. In particular, $A\vee X$ is $\Q$-bad.
\end{abstract}

\maketitle

\section{Completion, bad spaces, and the main theorem}
\label{sec:result}

Let \(R\) be a commutative ring with unit.  For a pointed simplicial set
\(K\), write \(RK\) for the underlying simplicial set of the
reduced free simplicial \(R\)-module on \(K\).  The unit and multiplication
of this monad give the standard coaugmented cosimplicial resolution
\[
 K\longrightarrow RK
 \;\substack{\longrightarrow\\[-.65ex]\longrightarrow}\;
 R^2K
 \;\substack{\longrightarrow\\[-.65ex]\longrightarrow\\[-.65ex]
 \longrightarrow}\;\cdots.
\]
After a Reedy fibrant replacement, the \emph{Bousfield--Kan
\(R\)-completion} is
\begin{equation}\label{eq:BK-definition}
 K^\wedge_R=\Tot\bigl(R^{\bullet+1}K\bigr)
 \simeq\operatorname*{holim}_{[n]\in\Delta}
        R^{n+1}K,
\end{equation}
and the coaugmentation induces the completion map
\(\eta_K:K\to K^\wedge_R\); for a space one applies this construction to a
pointed singular complex \cite[Chapters~I and~III]{BK}.  The completion
detects \(R\)-homology equivalences: a map \(f\) induces an isomorphism on
reduced homology with coefficients in \(R\) if and only if
\(f^\wedge_R\) is a weak equivalence \cite{BK}.

There is also an \(\infty\)-categorical universal characterization.  Let
\(\mathcal K(R)\) be the full subcategory of spaces spanned by the empty
space and by products of Eilenberg--Mac Lane spaces \(K(A,n)\), where
\(A\) is an \(R\)-module and \(n\geq0\).  A coaugmented endofunctor is called
\(\mathcal K(R)\)-preserving when its coaugmentation is an equivalence on
every object of \(\mathcal K(R)\).  Bousfield--Kan \(R\)-completion is the
terminal such coaugmented endofunctor on the \(\infty\)-category of
spaces.  Its structure as a coaugmented functor extends, uniquely up to a
contractible space of choices, to the terminal
\(\mathcal K(R)\)-preserving \(\infty\)-monad
\cite[Theorem~7.8 and Corollary~7.9]{CodensityCompletion}.

Completion is nevertheless not a localization on all spaces.  A
localization functor is homotopy idempotent, whereas in general the natural
maps
\[
 K^\wedge_R\longrightarrow (K^\wedge_R)^\wedge_R
\]
need not be weak equivalences.  This failure is measured by goodness.

A connected space \(Y\) is \emph{\(R\)-good} if
\(\eta_Y:Y\to Y^\wedge_R\) induces an isomorphism on reduced \(R\)-homology,
and it is \emph{\(R\)-bad} otherwise.  Equivalently, completion is
idempotent at an \(R\)-good space.  On the full subcategory of \(R\)-good
spaces it agrees with localization with respect to \(R\)-homology
\cite{BousfieldLocalization}.

The real projective plane is \(\Z\)-bad \cite[Chapter~VII]{BK}; a wedge of two circles is
\(\Z\)-bad \cite{BousfieldTowers}; and, for every prime \(p\), finite
wedges of at least two circles are \(\Z/p\)-bad
\cite{BousfieldPadic}.  Rational badness is much more elusive.  For a
connected nilpotent space, Bousfield--Kan \(\Q\)-completion agrees with
ordinary rationalization \cite{BK}.

For a long time, even the elementary space \(S^1\vee S^1\) was not known to be $\Q$-bad.  If \(F_r\) is a noncyclic finitely generated
free group, then
\[
 (K(F_r,1))^\wedge_{\Q}\simeq K(\widehat F_{r,\Q},1),
\]
where \(\widehat F_{r,\Q}\) is the rational pronilpotent completion.
Since \(H_2(F_r;\Q)=0\), it is enough to find a nonzero class in
\(H_2(\widehat F_{r,\Q};\Q)\), but infinite commutator series in the
completed group need not survive group-homology boundaries.  This
question, posed by Bousfield, was answered by Ivanov and Mikhailov: they
proved that the latter homology group has uncountable dimension
\cite{BousfieldTowers,IM}.  Thus every finite wedge
of at least two circles is \(\Q\)-bad; in particular, their work supplied the
first known finite \(\Q\)-bad space \cite{IM}.

Our result gives a broad class of $\Q$-bad spaces. From now on, write
$\eta_Y:Y\to\comp Y$ for the $\Q$-completion map,
$H_*(-)=H_*(-;\Q)$, and $\cont=2^{\aleph_0}$.

\begin{theorem}\label{thm:main}
Let $X$ be a connected based CW complex with
$\widetilde H_*(X)\neq0$, and set
\[
 m=\min\{j\geq1:H_j(X)\neq0\},\qquad Y=S^1\vee X.
\]
Then
\begin{equation}\label{eq:main}
 \dim_{\Q}\coker\bigl(H_{2m}(Y)
          \xrightarrow{\,H_{2m}(\eta_Y)\,}H_{2m}(\comp{Y})\bigr)
 \geq\cont.
\end{equation}
In particular, $Y$ is $\Q$-bad and $\dim_{\Q}H_{2m}(\comp{Y})\geq\cont$.
More generally, if $X$ is simply connected and the rational Hurewicz map
\[
 \pi_n(X)\otimes\Q\longrightarrow H_n(X)
\]
is nonzero for some $n\geq2$, then
\begin{equation}\label{eq:spherical-main}
 \dim_{\Q}\coker\bigl(H_{2n}(Y)
          \xrightarrow{\,H_{2n}(\eta_Y)\,}H_{2n}(\comp{Y})\bigr)
 \geq\cont.
\end{equation}
\end{theorem}

\begin{corollary}\label{cor:criterion}
For a connected based CW complex $X$, the wedge $S^1\vee X$ is
$\Q$-good if and only if $\widetilde H_*(X)=0$.
\end{corollary}

\begin{proof}
The theorem gives necessity. If $X$ is rationally acyclic, the collapse
$S^1\vee X\to S^1$ is a rational homology equivalence. Naturality of
completion, together with the goodness of $S^1$, then shows that
$S^1\vee X$ is $\Q$-good.
\end{proof}

\begin{corollary}[Wedges with nonzero first homology]
\label{cor:generalisation}
Let $A$ and $X$ be connected based CW complexes such that
$H_1(A)\neq0$ and $\widetilde H_*(X)\neq0$. Set
\[
 m=\min\{j\geq1:H_j(X)\neq0\},\qquad Z=A\vee X.
\]
Then
\begin{equation}\label{eq:generalised-main}
 \dim_{\Q}\coker\bigl(H_{2m}(Z)
       \xrightarrow{\,H_{2m}(\eta_Z)\,}H_{2m}(\comp Z)\bigr)
 \geq\cont.
\end{equation}
In particular, $Z$ is $\Q$-bad and
$\dim_{\Q}H_{2m}(\comp Z)\geq\cont$.
More generally, if $X$ is simply connected and the rational Hurewicz map
\[
 \pi_n(X)\otimes\Q\longrightarrow H_n(X)
\]
is nonzero for some $n\geq2$, then
\begin{equation}\label{eq:generalised-spherical}
 \dim_{\Q}\coker\bigl(H_{2n}(Z)
       \xrightarrow{\,H_{2n}(\eta_Z)\,}H_{2n}(\comp Z)\bigr)
 \geq\cont.
\end{equation}
\end{corollary}

The proof is a consequence of Theorem~\ref{thm:main} and naturality of
completion; it is given in Section~\ref{sec:generalisation}. No
finiteness or nilpotence hypothesis on $A$ is required.

The principal case is $m\geq2$, which includes simply connected spaces
$X$ in the theorem. To explain the argument, take a minimal Lie model
$(\LL(V),d)$ of $X$ as in Lemma~\ref{lem:model} and
\eqref{eq:wedge-model}, where
$V_{q-1}\cong\widetilde H_q(X)$ and the differential is decomposable.
Wedging with a circle adjoins a degree-zero cycle $a$ and requires
completion by bracket length, giving the model of~\eqref{eq:L}:
\[
 L=(\widehat\LL(a,V),d),\qquad da=0.
\]
Choose a nonzero bottom-degree generator $b\in V_{m-1}$. Degree
considerations give $db=0$, as checked at the beginning of the proof in
Section~\ref{sec:main-proof}. Its iterated brackets
$b_j=\ad_a^j(b)$ all have degree $m-1$, but their bracket lengths tend
to infinity. Completion therefore allows arbitrary series in the $b_j$
within this single homological degree. Sending these iterates to powers
of a formal variable gives the quotient constructed in
\eqref{eq:quotient}, with target defined in~\eqref{eq:target}:
\[
 L\twoheadrightarrow G_m=\Q a\ltimes\Q[[t]][m-1],
 \qquad b_j\longmapsto t^j,\qquad [a,f(t)]=tf(t).
\]
Its positive ideal is abelian, while the action of $a$ still records all
the iterates of $b$.

Upon realization of complete dgl's, Lemma~\ref{lem:split}, specialized
in~\eqref{eq:target-borel}, identifies the split extension defining
$G_m$ with a fibration over $K(\Q,1)$ having fiber $K(\Q[[t]],m)$ and
monodromy $f(t)\mapsto e^{qt}f(t)$ for $q\in\Q$. By
Lemma~\ref{lem:target-homology}, its degree-$2m$ homology is the
coinvariant quotient of the quadratic fiber homology. The symmetric
and exterior power-series detector of Lemma~\ref{lem:detector}, modeled
on \cite[Proposition~2.1]{IM}, detects a subspace of dimension $\cont$
in this quotient. The telescoping bracket identity~\eqref{eq:telescope}
and the Bernoulli correction in~\eqref{eq:series-lifts} then produce the
cycles~\eqref{eq:lifted-cycle}; Proposition~\ref{prop:lifting} shows
that their images contain this subspace.

Finally, the compatibility square~\eqref{eq:wedge-unit-square} shows
that the restriction to $X$ of the composite from the original wedge
factors through $K(\Q,m)$, yielding the at-most-one-dimensional image
in~\eqref{eq:small-image}. Comparing this with the
continuum-dimensional image~\eqref{eq:large-image} proves the cokernel
bound in Section~\ref{sec:main-proof}, regardless of the size of
$H_{2m}(X)$.
The same construction only needs a cycle with a nonzero indecomposable
component, not necessarily one in the bottom nonzero degree. In a
minimal Lie model, such cycles represent nonzero rational spherical
homology classes. This gives the more general bound
\eqref{eq:spherical-main}, proved at the end of
Section~\ref{sec:main-proof}.

Section~\ref{sec:chains} identifies the relevant split Lie realizations,
and Section~\ref{sec:chain-model} describes their chain models.
Section~\ref{sec:detector} computes the quotient
homology and constructs the detector, and Section~\ref{sec:lifts} gives
the lifted cycles. These ingredients prove the case $H_1(X)=0$ in
Section~\ref{sec:main-proof}. Section~\ref{sec:examples} reduces the
remaining case to the theorem of Ivanov and Mikhailov and records some
consequences.
Section~\ref{sec:generalisation} then replaces the circle by an arbitrary
connected based CW complex with nonzero first rational homology.

\section{Presenting the completion}
\label{sec:chains}


All Lie algebras and tensor products are over $\Q$. Differential graded
Lie algebras (dgls) are homologically graded, with differential of degree
$-1$. A dgl is \emph{positive} if it is concentrated in degrees at least
one. A \emph{complete dgl} is a dgl $M$ equipped with a decreasing
filtration by differential ideals satisfying
\[
 F^1M=M,\qquad [F^pM,F^qM]\subseteq F^{p+q}M,\qquad
 M\cong\varprojlim_r M/F^rM.
\]
Morphisms preserve each filtration level and are therefore continuous.
In particular, all maps used below preserve the specified filtration
indices. The free Lie algebra completed by bracket length is
\[
 \widehat\LL(V)=\varprojlim_r\LL(V)/\LL^{\geq r}(V).
\]
This filtration controls convergence of infinite Lie series. It is
distinct from homological degree, which may remain fixed as bracket
length increases.

We use the complete Lie realization of \cite{LieModels}, defined by
\begin{equation}\label{eq:Lie_realization}
 \real M_n=\operatorname{Hom}_{\mathbf{cdgl}}(\mathfrak L_n,M),
 \qquad
 \mathfrak L_n=(\widehat\LL(s^{-1}C_*(\Delta^n;\Q)),d).
\end{equation}
Here $C_*$ denotes normalized simplicial chains, $\mathfrak L_\bullet$
is the cosimplicial complete simplex model, and its filtration is bracket
length. Morphisms into $M$ preserve the specified filtration levels.
We use the standard sign convention for desuspension:
\(d_0(s^{-1}c)=-s^{-1}\partial c\).
Its homotopy groups, for $M$ nonnegatively graded, are
\begin{equation}\label{eq:homotopy}
 \pi_n\real M\cong H_{n-1}(M),\qquad n\geq1,
\end{equation}
with the Baker--Campbell--Hausdorff group structure on $H_0(M)$
(see \cite[Theorem~7.18 and \S 4.2]{LieModels}).
If $L_Y$ is a complete Lie model of a connected space $Y$, the unit map
identifies, naturally up to homotopy, with completion:
\begin{equation}\label{eq:completion}
 Y\longrightarrow\real{L_Y}\simeq\comp Y
 \qquad\text{\cite[Theorem 0.1]{Completion}.}
\end{equation}
This theorem is a key ingredient in the argument of this paper.

More explicitly, for a reduced simplicial set $K$ with unique vertex
$a_0$, put $\mathfrak L^a(K)=\mathfrak L(K)/(a_0)$, where
$\mathfrak L(K)$ is its global complete Lie model. This gives the
reduced/connected adjunction
\begin{equation}\label{eq:reduced-adjunction}
 \mathfrak L^a:\mathbf{sSet}_0\rightleftarrows
 \mathbf{cdgl}_0:\real{-}.
\end{equation}
The subscript $0$ means reduced simplicial sets on the left and
nonnegatively graded complete dgls on the right. The unit in
\eqref{eq:completion} is the unit of this adjunction
\cite[Section 1]{Completion}.

\medskip

We shall use the Maurer--Cartan description of the realization:
\[
 \MC_\bullet(M)=\MC\bigl(\Omega_\bullet\widehat\otimes M\bigr),
 \qquad
 \Omega_\bullet=A_{\mathrm{PL}}(\Delta^\bullet),
\]
where $\Omega_\bullet^k$ is placed in homological degree $-k$ and
the tensor product is completed in the filtration of $M$:
\[
 \Omega_\bullet\widehat\otimes M
 =\varprojlim_r(\Omega_\bullet\otimes M/F^rM).
\]
Here $\MC(A)$ consists of elements $\omega\in A_{-1}$ satisfying
$d\omega+\tfrac12[\omega,\omega]=0$.
Explicitly,
\[
 \MC_n(M)=
 \left\{\omega\in(\Omega_n\widehat\otimes M)_{-1}:
 (d_{\Omega}+d_M)\omega+\tfrac12[\omega,\omega]=0\right\}.
\]
Writing
\[
 \Omega_n=
 \frac{\Q[t_0,\ldots,t_n,dt_0,\ldots,dt_n]}
 {\bigl(\sum_i t_i-1,\,\sum_i dt_i\bigr)},
 \qquad |t_i|=0,\quad |dt_i|=1,
\]
the simplicial operator associated to an order-preserving map
$\vartheta:[k]\to[n]$ is
\begin{equation}\label{eq:mc-simplicial-operators}
\begin{aligned}
 \vartheta^*:\MC_n(M)&\longrightarrow\MC_k(M),\\
 \omega&\longmapsto
   (\vartheta^*\widehat\otimes\mathrm{id}_M)(\omega),\\
 \vartheta^*(t_i)&=\sum_{\vartheta(j)=i}t_j,
 & \vartheta^*(dt_i)&=\sum_{\vartheta(j)=i}dt_j.
\end{aligned}
\end{equation}

An $n$-simplex $\omega$ has the following explicit form:
\begin{equation}\label{n-simplex_of_MC}\omega
=
X_{\varnothing}(t)
+\sum_{i=1}^{n}dt_i\,X_i(t)
+\sum_{1\leq i<j\leq n}dt_i\,dt_j\,X_{ij}(t)
+\cdots
+dt_1\cdots dt_n\,X_{1\cdots n}(t),\end{equation}
where, for \(I=\{i_1<\cdots<i_p\}\),
\[
X_I(t)\in
\mathbb Q[t_1,\ldots,t_n]\widehat\otimes M_{p-1}.
\]
Since \(M_{-1}=0\), we have \(X_{\varnothing}=0\) for all $\omega$.

\begin{lemma}[Maurer--Cartan comparison]
\label{lem:mc-comparison}
Let \(M\) be a nonnegatively graded complete filtered dgl. There
are natural simplicial maps
\[
 \real M
 \xrightarrow{\ \iota_M\ }
 \MC_\bullet(M)
 \xrightarrow{\ \rho_M\ }
 \real M,
 \qquad
 \MC_\bullet(M)
 =
 \MC\bigl(\Omega_\bullet\widehat\otimes M\bigr),
\]
such that
\[
 \rho_M\iota_M=\operatorname{id}_{\real M},
\]
and both maps are weak equivalences. Naturality is with respect
to filtration-preserving dgl morphisms. In particular, both maps
commute with every filtration-preserving dgl automorphism of \(M\).
\end{lemma}

The proof is given in Appendix~\ref{app:comparison-results}.

\medskip
Let $\Gamma=(\Q,+)$ be a discrete group. Suppose that $M$ is a positive
complete dgl, and that $D$ is a degree-zero chain derivation with
\begin{equation}\label{eq:raising}
 D(F^rM)\subseteq F^{r+1}M.
\end{equation}
By~\eqref{eq:homotopy}, $\real{M}$ is simply connected.
The exponential $T_q=e^{qD}$ is a filtered chain automorphism for every $q\in\Q$.
Set
\[
 L_D=\Q a\ltimes_D M,\qquad |a|=0,\quad da=0,\quad[a,x]=Dx,
\]
with $F^1L_D=L_D$ and $F^rL_D=F^rM$ for $r\geq2$.
These formulas define a complete dgl filtration.

The projection $L_D\to\Q a$ and its section give maps over
$\real{\Q a}\simeq K(\Gamma,1)$. The following lemma identifies the
realization of $L_D\to \Q a$: the infinitesimal action $D$ integrates to
the strict action $q\mapsto e^{qD}$ on the fiber. We write
$E\Gamma\times_\Gamma F$ for the Borel construction, or homotopy
quotient, of a simplicial $\Gamma$-space $F$.

\begin{lemma}[Split realization]\label{lem:split}
There is a natural weak equivalence
\begin{equation}\label{eq:borel}
 \real{L_D}\simeq E\Gamma\times_\Gamma\real M,
 \qquad q\cdot x=\real{e^{qD}}(x),
\end{equation}
over $B\Gamma$. In particular, $\real M$ has the homotopy type of the
universal cover of $\real{L_D}$.
\end{lemma}

\begin{proof}

By Lemma~\ref{lem:mc-comparison}, we may work with the completed
Maurer--Cartan realizations. An \(n\)-simplex of
\(\MC_\bullet(L_D)\) (see~\eqref{n-simplex_of_MC}) has a unique expression
\[
 \omega=\alpha a+\eta,
\]
where
\[
 \alpha=\sum_{i=1}^n f_i(t)\,dt_i\in\Omega_n^1,
 \qquad
 \eta=\Big(\sum_{i<j}dt_i\,dt_j\,X_{ij}(t)+\cdots\Big) \in(\Omega_n\widehat\otimes M)_{-1}.
\]
Its Maurer--Cartan equation is
\[
 d\alpha=0,\qquad
 d\eta+\alpha D\eta+\tfrac12[\eta,\eta]=0.
\]
Since $\alpha$ is closed, there is a unique $h\in\Omega_n^0$ such that
$dh=\alpha$ and $v_0^*h=0$, where $v_0:[0]\to[n]$ selects the initial vertex and $v_0^*$ is the pullback described in
\eqref{eq:mc-simplicial-operators}. The gauge transformation\footnote{This
is the dgl analogue of changing the trivialization of a connection: \(A\longmapsto gAg^{-1}-(dg)g^{-1},\; g=e^x.\) Indeed, if
$d_\omega=d+\ad_\omega$, then
$d_{x\mathbin{\mathcal G}\omega}=e^{\ad_x}d_\omega e^{-\ad_x}$.}
\[
 x\mathbin{\mathcal G}\omega
 =e^{\ad_x}\omega-
   \frac{e^{\ad_x}-1}{\ad_x}(dx)\footnote{We adopt the notation of \cite[p.~4]{LieModels}.}
\]
with $x=ha$ sends $\omega$ to $\xi=e^{hD}\eta\in\MC_n(M)$.
Thus the correspondence \begin{equation}\label{gauged_coordinates}\omega\leftrightarrow(h,\xi)\end{equation} is a bijection, with inverse
$(h,\xi)\mapsto(dh)a+e^{-hD}\xi=\omega$. Every series here converges by the condition
\eqref{eq:raising}.
\smallskip
For a simplicial operator $\vartheta:[k]\to[n]$, in these coordinates
$\omega$ transforms as
\begin{equation}\label{eq:coordinates}
 (h,\xi)\longmapsto
 \bigl(\vartheta^*h-c,\ e^{-cD}\vartheta^*\xi\bigr),
 \qquad c=h(v_{\vartheta(0)}).
\end{equation}
Give $\mathcal H_n=\{h\in\Omega_n^0:h(v_0)=0\}$ the following simplicial structure. For a simplicial operator \(\vartheta:[k]\to[n]\), define
\[
 \vartheta^\sharp:\mathcal H_n\longrightarrow\mathcal H_k,
 \qquad
 \vartheta^\sharp h
 =
 \vartheta^*h-h(v_{\vartheta(0)}).
\]

The map
\[
 \epsilon:\mathcal H_\bullet\longrightarrow B\Gamma,\qquad
 h\longmapsto
 \bigl(h(v_0)-h(v_1),\ldots,h(v_{n-1})-h(v_n)\bigr)
\]
is a trivial Kan fibration. Under the simplicial identification \(h\mapsto (dh)a\) of
\(\mathcal H_\bullet\) with \(\MC_\bullet(\Q a)\), the map $\epsilon$ is given, up to a sign, by
\[
 \alpha a\longmapsto
 \left(
  \int_{[v_0,v_1]}\alpha,\ldots,
  \int_{[v_{n-1},v_n]}\alpha
 \right),
\]
the standard simplicial de~Rham integration
\[
 \MC_\bullet(\Q a)\longrightarrow
 K(\Q,1)=B\Gamma.
\]
It is a weak equivalence by the abelian Dold--Kan identification
\cite[Proposition~5.1 and Corollary~5.11]{Getzler}. The map
\(\epsilon\) differs from it only by the automorphism \(q\mapsto-q\)
of \(B\Gamma\). The map \(\epsilon\) is degreewise surjective\footnote{Indeed, given
\((q_1,\ldots,q_n)\in(B\Gamma)_n=\Q^n\), set
\[
 \lambda_0=0,\qquad
 \lambda_i=-\sum_{j=1}^i q_j,
\]
and define the affine polynomial
\[
 h=\sum_{i=0}^n\lambda_i t_i\in\Omega_n^0.
\]
Then \(h(v_0)=0\) and
\[
 h(v_{i-1})-h(v_i)=q_i,
\]
so \(\epsilon_n(h)=(q_1,\ldots,q_n)\).}. Moreover, $\epsilon$ is a morphism of simplicial
vector spaces, and every degreewise surjective morphism of simplicial
vector spaces is a Kan fibration.

For the action $q\cdot\xi=e^{qD}\xi$, we use the inhomogeneous
simplicial Borel model $E\Gamma\times_\Gamma\MC_\bullet(M)$ whose $n$-simplices are elements of
$\MC_n(M)\times\Gamma^n$, written $(\xi;q_1,\ldots,q_n)$, and whose
operators are
\[
\begin{aligned}
 d_0(\xi;q_1,\ldots,q_n)
   &=(e^{q_1D}d_0\xi;q_2,\ldots,q_n),\\
 d_i(\xi;q_1,\ldots,q_n)
   &=(d_i\xi;q_1,\ldots,q_i+q_{i+1},\ldots,q_n)
     &&(0<i<n),\\
 d_n(\xi;q_1,\ldots,q_n)
   &=(d_n\xi;q_1,\ldots,q_{n-1}),\\
 s_i(\xi;q_1,\ldots,q_n)
   &=(s_i\xi;q_1,\ldots,q_i,0,q_{i+1},\ldots,q_n)
     &&(0\leq i\leq n).
\end{aligned}
\]
Its projection to the inhomogeneous bar model
$(B\Gamma)_n=\Gamma^n$ forgets $\xi$.
Using the coordinates~\eqref{gauged_coordinates}, the formula
\[
\Phi_n(\omega)
=
\bigl(h,(\xi;q_1,\ldots,q_n)\bigr),\qquad (q_1,\ldots,q_n)=\epsilon_n(h),\]
is compatible with simplicial operators and gives an isomorphism of simplicial sets
\[
 \MC_\bullet(L_D)\cong
 \mathcal H_\bullet\times_{B\Gamma}
 \bigl(E\Gamma\times_\Gamma\MC_\bullet(M)\bigr).
\]
Finally, we obtain the following zigzag of weak equivalences
\[\begin{tikzcd}
	{ \real{L_D}} & {\MC_\bullet(L_D) } & { \mathcal H_\bullet\times_{B\Gamma}  \bigl(E\Gamma\times_\Gamma\MC_\bullet(M)\bigr).} \\
	&& { E\Gamma\times_\Gamma\MC_\bullet(M)} \\
	&& { E\Gamma\times_\Gamma\real M}
	\arrow["{\overset{(1)}\sim}", hook, from=1-1, to=1-2]
	\arrow["\cong", from=1-2, to=1-3]
	\arrow["{\overset{(2)}\sim}", from=1-3, to=2-3]
	\arrow["{\overset{(4)}\sim}"', shift right=5, two heads, from=2-3, to=3-3]
	\arrow["{\overset{(3)}{\sim}}"', hook, from=3-3, to=2-3]
\end{tikzcd}\]
Map $(1)$ is the weak equivalence $\iota_{L_D}$ of
Lemma~\ref{lem:mc-comparison}, identified with the inclusion of the
gauge-fixed nerve into the Maurer--Cartan realization. Projection $(2)$ is the pullback of $\epsilon$, so
it is a trivial Kan fibration and thus a weak equivalence.
Maps $(3)$ and $(4)$ are induced by the maps in
Lemma~\ref{lem:mc-comparison}. Both are weak equivalences because
the Borel construction preserves equivariant weak equivalences;
map $(4)$ is a left inverse of map $(3)$.
We identify the edge coordinates of \(B\Gamma\) by
\[
 q_i
 =
 -\int_{[v_{i-1},v_i]}\alpha
 =
 h(v_{i-1})-h(v_i).
\]
In particular, for the zeroth face,
\(q_1=-h(v_1)=-c\), and hence the fiber coordinate transforms as
\[
 e^{-cD}d_0\xi=e^{q_1D}d_0\xi.
\]
Thus, with the standard inhomogeneous Borel convention, the resulting
monodromy action on realization is \(q\mapsto \real{e^{qD}}: \real{M}\to \real{M}\). Since $\real{M}$ is simply connected, the map
$\real{M}\to\real{L_D}$ induced by the inclusion $M\hookrightarrow L_D$
has the homotopy type of the universal covering map of $\real{L_D}$.
\end{proof}

\begin{lemma}[Wedge model]\label{lem:model}
Let $X$ be connected with $H_1(X)=0$. There is a minimal complete
Lie model
\begin{equation}\label{eq:wedge-model}
 M_X=(\LL(V),d),\qquad
 V_{q-1}\cong\widetilde H_q(X),\qquad dV\subseteq\LL^{\geq2}(V),
\end{equation}
with $V$ positive. A complete Lie model of $Y=S^1\vee X$ is
\begin{equation}\label{eq:L}
 L=(\widehat\LL(a,V),d),\qquad |a|=0,\quad da=0,
\end{equation}
whose differential restricts to that of $M_X$.
The equivalences $\real{M_X}\simeq\comp X$ and
$\real L\simeq\comp Y$ are compatible, up to homotopy, with the
inclusions of the wedge summands and their completion maps.
\end{lemma}

\begin{proof}
Choose a reduced simplicial model $K_X$ of $X$, form
$\mathfrak L^a(K_X)$, and take its minimal complete Lie model $M_X$.
The minimal-model theorem identifies the suspended indecomposables
with reduced rational homology
\cite[Proposition 8.35]{LieModels}, \cite[Section 1]{Completion}.
Thus $H_1(X)=0$ makes $V$ positive, and completion of $\LL(V)$ by
bracket length changes no homogeneous degree.

Use the model structure on complete dgls transferred from simplicial
sets, and its reduced/connected adjunction
\eqref{eq:reduced-adjunction}; see \cite[Corollary 8.2 and
Proposition 8.3]{LieModels} and \cite[Section 1]{Completion}.
All simplicial sets are cofibrant, so their wedge is a homotopy
coproduct. The left derived reduced model carries it to the derived
completed coproduct. The free models $\Q a$ and $M_X$ are cofibrant
\cite[Section 1]{Completion}, so this coproduct is exactly
$\Q a\,\widehat\amalg\,M_X=(\widehat\LL(a,V),d)$.
The differential on $V$ consequently contains no $a$.
Choose these model comparisons as a diagram containing the two
coproduct injections. Naturality of the adjunction unit and
\eqref{eq:completion} then give the homotopy-commutative square
\begin{equation}\label{eq:wedge-unit-square}
 \begin{array}{ccc}
 X&\longrightarrow&Y=S^1\vee X\\[2pt]
 \eta_X\big\downarrow&&\big\downarrow\eta_Y\\[2pt]
 \real{M_X}&\longrightarrow&\real L,
 \end{array}
\end{equation}
where the vertical maps use the equivalences with rational completion.
The analogous square holds for the circle summand. These squares give
the required compatibility with the original spaces.
\end{proof}

\begin{corollary}[Universal cover of the completed wedge]
\label{cor:wedge-universal-cover}
Let $X$ be connected with $H_1(X)=0$, and choose
the models of Lemma~\ref{lem:model}.  Let
\[
 p:L=(\widehat\LL(a,V),d)\longrightarrow\Q a,
 \qquad p(a)=a,\quad p(V)=0,
\]
and set
\[
 N_X=\ker p,\qquad D=\ad_a|_{N_X}.
\]
Then $N_X$ is a positive complete dgl, $D$ satisfies
\[
 D(F^rN_X)\subseteq F^{r+1}N_X,
\]
and the wedge model is the split extension
\[
 L=L_D=\Q a\ltimes_D N_X.
\]
Consequently there is a homotopy fibration sequence
\begin{equation}\label{eq:wedge-universal-cover}
 \real{N_X}\longrightarrow
 \comp{(S^1\vee X)}\simeq\real{L_D}
 \longrightarrow K(\Gamma,1),
 \qquad \Gamma=(\Q,+),
\end{equation}
whose monodromy is
\[
 q\longmapsto\real{e^{qD}}.
\]
In particular, $\real{N_X}$ has the homotopy type of the universal
cover of $\comp{(S^1\vee X)}$.  The inclusion $M_X\hookrightarrow N_X$
also induces a lift
\[
 \comp X\simeq\real{M_X}\longrightarrow\real{N_X}
\]
of the completed wedge inclusion
$\comp X\to\comp{(S^1\vee X)}$.
\end{corollary}

\begin{proof}
The kernel $N_X$ is the closed differential Lie ideal consisting of
convergent series of bracket words that contain at least one letter
from $V$.  Since $V$ is positive and $|a|=0$, it is positive.  Every
element of $L$ has a unique expression $qa+x$ with $q\in\Q$ and
$x\in N_X$, and the inner derivation $D=\ad_a$ preserves $N_X$ and
commutes with the differential.  With the filtration induced by bracket
length, applying $D$ adds one occurrence of $a$, and hence
$D(F^rN_X)\subseteq F^{r+1}N_X$.  Thus
$L=\Q a\ltimes_DN_X$ satisfies the hypotheses of
Lemma~\ref{lem:split}.

That lemma gives a weak equivalence over $B\Gamma=K(\Gamma,1)$,
\[
 \real L\simeq E\Gamma\times_\Gamma\real{N_X},
\]
with monodromy $q\mapsto\real{e^{qD}}$, and identifies
$\real{N_X}$ with the universal cover of $\real L$.
Lemma~\ref{lem:model} identifies $\real L$ with
$\comp{(S^1\vee X)}$ and $\real{M_X}$ with $\comp X$, compatibly with
the wedge inclusion.  Since the inclusion $M_X\hookrightarrow L$
factors through $N_X$, the final assertion follows as well.
\end{proof}

\section{A chain model for the homology of the completion}
\label{sec:chain-model}
We now pass from this geometric description to chains. Throughout this
section, $M$, $D$, $L_D$, and $\Gamma$ are as in
Lemma~\ref{lem:split}: $M$ is a positive complete dgl, and $D$ is a
degree-zero chain derivation satisfying~\eqref{eq:raising}.


For a positive dgl $M$, let $sM$ denote its suspension,
$(sM)_q=M_{q-1}$, and let
\[
 \CE(M)=\bigl(S^c(sM),\partial\bigr),\qquad
 S^c(sM)=\bigoplus_{k\geq0}\Sym^k_{\mathrm{gr}}(sM),
\]
be its Chevalley--Eilenberg chain coalgebra. The coproduct is the
graded unshuffle coproduct. The differential is the coderivation whose
projection to $sM$ is given by
\begin{equation}\label{eq:CE}
 sx\longmapsto-s(dx),\qquad
 (sx)(sy)\longmapsto(-1)^{|x|}s[x,y],
\end{equation}
and is zero on all other weights. In particular,
$(sx)(sy)=(-1)^{(|x|+1)(|y|+1)}(sy)(sx)$.
This is the ordinary, uncompleted chain coalgebra: each chain is a
finite sum of finite products, although its entries may be complete
Lie series.

The number of suspended factors is the \emph{weight} of a monomial,
whereas its homological degree is
\[
 \bigl|(sx_1)\cdots(sx_k)\bigr|
   =\sum_{i=1}^k(|x_i|+1).
\]
The two parts of \eqref{eq:CE} encode the internal differential and the
Lie bracket. In particular, even when $dx=dy=0$, the product
$(sx)(sy)$ need not be a Chevalley--Eilenberg cycle: its boundary is the
suspended bracket $(-1)^{|x|}s[x,y]$.

The automorphisms $T_q$ act diagonally on $\CE(M)$.
Use the normalized group bar complex and the total differential
\begin{equation}\label{eq:total}
 B_{p,k}(M)=C_p(\Gamma;\CE(M)_k),\qquad
 \partial_{\mathrm{tot}}=\partial_{\mathrm{bar}}+(-1)^p\partial,
\end{equation}
where $\partial_{\mathrm{bar}}([q]\otimes w)=(T_q-1)w$.
The index $p$ is the bar degree and $k$ is the
Chevalley--Eilenberg degree; total degree is their sum. A bar one-chain
therefore provides a boundary measuring the difference between a chain
and its translate by monodromy. In Section~\ref{sec:lifts}, we use the next lemma to construct a
continuum of cycles in the total complex~\eqref{eq:total}.

\begin{lemma}[Equivariant chain model]\label{lem:chains}
There is a natural isomorphism
\begin{equation}\label{eq:chain-model}
 H_*(\real{L_D})\cong H_*\bigl(\Tot B_{*,*}(M)\bigr).
\end{equation}
It is compatible with morphisms of split extensions that are the
identity on $\Q a$, and with the inclusions of the fibers into their
Borel constructions.
\end{lemma}

\begin{proof}
We construct the chain comparison using the reduced model in
\eqref{eq:reduced-adjunction}. First let $\gamma$ be the lower central
filtration of $M$. Positivity gives
$(\gamma^rM)_j=0$ for $r>j$, so $M$ is also complete for $\gamma$.
Since $\gamma^rM\subseteq F^rM$, the identity
$\iota:(M,\gamma)\to(M,F)$ is a filtered quasi-isomorphism and upon realization
it becomes a weak equivalence. It commutes with every $T_q$ (see Lemma~\ref{lem:exponential-action}). Write
$K=\real{(M,\gamma)}$. Since $M$ is positive, $K$ has a unique vertex
and no nondegenerate edges. Consequently its reduced model is the
ordinary positive free dgl
\[
 R_K=\mathfrak L^a(K)=(\LL(W_K),d),\qquad
 W_K=s^{-1}\widetilde C_*(K;\Q).
\]
Completion changes no homogeneous degree here. The linear part $d_0$
of $d$ is the desuspended chain differential.

The adjunction counit $\varepsilon_M:R_K\to(M,\gamma)$ is a natural
quasi-isomorphism. To see this without a finite-type hypothesis, note
that $K$ is simply connected and all its homotopy groups are rational
vector spaces, by \eqref{eq:homotopy}. Its $\Q$-completion map is therefore a weak equivalence. By \eqref{eq:completion}, this map is the
adjunction unit $\eta_K:K\to\real{R_K}$. The triangle identity
\[K\xrightarrow{\eta_K}\real{R_K}
\xrightarrow{\real{\varepsilon_M}}K
=\mathrm{id}_K\]
and the two-out-of-three property show that
$\real{\varepsilon_M}$ is a weak equivalence, and
\eqref{eq:homotopy} implies that $\varepsilon_M$ is a quasi-isomorphism.

There is a natural chain projection
\[
 p_K:\CE(R_K)\longrightarrow
 \Q\oplus s(R_K/[R_K,R_K])
 \cong C_*(K;\Q).
\]
Indeed, let
\[
 \pi:R_K\longrightarrow R_K/[R_K,R_K]
\]
be the projection onto Lie indecomposables. Since
\[
 R_K=(\LL(W_K),d)
\]
is free as a graded Lie algebra, restriction to the generators gives
\[
 R_K/[R_K,R_K]\cong W_K.
\]
Writing \(d=d_0+d_{\geq2}\), the decomposable part \(d_{\geq2}\)
vanishes on indecomposables, so the induced differential on the
quotient is \(d_0\). Hence
\[
 \Q\oplus s(R_K/[R_K,R_K])
 \cong \Q\oplus(sW_K,-sd_0)
 \cong C_*(K;\Q),
\]
because \(W_K=s^{-1}\widetilde C_*(K;\Q)\) and
\[
 d_0(s^{-1}c)=-s^{-1}\partial c.
\]

Define \(p_K\) as the composite
\[
 \CE(R_K)\longrightarrow\Q\oplus sR_K
 \xrightarrow{\operatorname{id}_{\Q}\oplus s\pi}
 \Q\oplus s(R_K/[R_K,R_K]),
\]
where the first map is projection onto CE weights zero and one.
Equivalently,
\[
 p_K(1)=[v_0],\qquad
 p_K(sx)=s\pi(x),\qquad
 p_K\bigl(\Sym_{\mathrm{gr}}^{\geq2}(sR_K)\bigr)=0.
\]
In particular, for \(c\in\widetilde C_*(K;\Q)\),
\[
 p_K\bigl(s(s^{-1}c)\bigr)=c.
\]

The map \(p_K\) commutes with the differentials. On CE weight one,
\[
 p_K\partial(sx)
 =-s\pi(dx)
 =-sd_0\pi(x)
 =\partial p_K(sx).
\]
On higher weights, the internal part of the CE differential preserves
weight, while the bracket part lowers it by one. Thus only weight two
could produce a term of weight one, but that term is a multiple of
\(s[x,y]\), which is killed because
\(\pi([x,y])=0\). Hence
\[
 p_K\partial=\partial p_K.
\]

The construction is natural in \(K\). By the free-Lie projection
theorem \cite[Lemma~2.2(ii)]{LieModels}, the corresponding projection
onto \((sW_K,-sd_0)\) is a quasi-isomorphism. Adjoining the
weight-zero copy of \(\Q\) therefore gives the natural
quasi-isomorphism
\[
 p_K:\CE(R_K)\xrightarrow{\simeq}C_*(K;\Q).
\]


Finally, applying $\CE$ to $\varepsilon_M$ preserves its
quasi-isomorphism. The increasing CE-weight filtration has graded
pieces $\Sym^k_{\mathrm{gr}}(sR_K)$ and
$\Sym^k_{\mathrm{gr}}(sM)$; symmetric powers preserve
quasi-isomorphisms over $\Q$, and positivity bounds the weights in
each degree. We therefore obtain the explicit natural zigzag
\begin{equation}\label{eq:chain-zigzag}
 C_*(\real M;\Q)
 \xleftarrow[\simeq]{C_*(\real\iota)} C_*(K;\Q)
 \xleftarrow[\simeq]{p_K} \CE(R_K)
 \xrightarrow[\simeq]{\CE(\varepsilon_M)} \CE(M).
\end{equation}
All complexes are ordinary nonnegative chain complexes, with their
degree-zero copy of $\Q$ included. The reduced model, its counit,
and the indecomposable projection are functorial. Every arrow therefore
commutes with $T_q$ and with each filtered morphism intertwining the
derivations. In particular, \eqref{eq:chain-zigzag} is a zigzag of
$\Q[\Gamma]$-complexes.

By Lemmas~\ref{lem:split} and
\ref{lem:equivariant-bar-comparison}, there is a natural
quasi-isomorphism
\[
 C_*(\real{L_D};\Q)
 \simeq
 \Tot C_*\bigl(\Gamma;C_*(\real M;\Q)\bigr).
\]
Applying the total group-bar construction
\(\Tot C_*(\Gamma;-)\) to the equivariant zigzag
\eqref{eq:chain-zigzag} preserves quasi-isomorphisms by
Lemma~\ref{lem:equivariant-bar-comparison}. Hence
\[
 C_*(\real{L_D};\Q)
 \simeq
 \Tot C_*\bigl(\Gamma;\CE(M)\bigr)
 =
 \Tot B_{*,*}(M).
\]
Taking homology proves \eqref{eq:chain-model}. Naturality and
compatibility with the fiber inclusions follow from the corresponding
assertions in Lemmas~\ref{lem:split} and
\ref{lem:equivariant-bar-comparison}.
\end{proof}

\section{The power-series quotient}
\label{sec:detector}


Fix $m\geq2$ and put $P=\Q[[t]]$. Regard $P[m-1]$ as an abelian dgl
with zero differential, and define
\begin{equation}\label{eq:target}
 G_m=\Q a\ltimes P[m-1],\qquad [a,f]=tf,
 \qquad F^rP=t^{r-1}P\quad(r\geq1).
\end{equation}
The filtration on $G_m$ is the split filtration of
Section~\ref{sec:chains}. By Lemma~\ref{lem:split}, its realization is
\begin{equation}\label{eq:target-borel}
 \real{G_m}\simeq E\Gamma\times_\Gamma K(P,m),\qquad
 q\cdot f=e^{qt}f.
\end{equation}
The first quadratic products in the fiber occur in degree $2m$.
Their symmetry is determined by the degree $m$ of each factor, which
is one greater than the Lie degree $m-1$.
For a vector space $W$, write
\[
 A_m^2(W)=
 \begin{cases}
  \Sym^2(W),&m\text{ even},\\
  \Lambda^2(W),&m\text{ odd},
 \end{cases}
\]
and denote its product by $f\star g$.
For a $\Gamma$-module $A$, write
\[
 A_\Gamma=A/\operatorname{span}_{\Q}
       \{q\cdot x-x:q\in\Gamma,\ x\in A\}
\]
for its coinvariants. These are the relations imposed on fiber homology
by transport around the base.
For the abelian dgl $P[m-1]$, we use the fiber-chain comparison
\eqref{eq:chain-zigzag} to identify its degree-$2m$ homology with
$A_m^2(P)$: the CE monomial $(sf)(sg)$ corresponds to $f\star g$.
This fixes the normalization used in the following lemma and in the
lifting proposition.

\begin{lemma}\label{lem:target-homology}
The fiber inclusion in \eqref{eq:target-borel} induces an isomorphism
\begin{equation}\label{eq:target-homology}
 A_m^2(P)_\Gamma\xrightarrow{\ \cong\ }H_{2m}(\real{G_m}).
\end{equation}
\end{lemma}

\begin{proof}
For every $\Q[\Gamma]$-module $A$, one has $H_p(\Gamma;A)=0$ for
$p\geq2$. Indeed, $\Gamma=\bigcup_{r\geq1}(1/r!)\Z$, its bar complex
is the filtered union of the bar complexes of these cyclic subgroups,
and filtered colimits are exact. Each infinite cyclic subgroup has
homological dimension one.

The rational homology of $K(P,m)$ is the free graded-commutative
algebra on $P$ in degree $m$. This statement holds for arbitrary $P$
by passage to the filtered union of its finite-dimensional subspaces.
Thus its groups in degrees $2m$ and $2m-1$ are $A_m^2(P)$ and zero,
respectively. In the Serre spectral sequence of
\eqref{eq:target-borel}, the only nonzero term of total degree $2m$
is consequently $E^2_{0,2m}=A_m^2(P)_\Gamma$, and no differential
can enter or leave it.
\end{proof}


We focus on classes represented by \(f\star1\). Replacing \(t\) by \(-t\) sends \(e^{qt}\) to its inverse, which lets us encode the coinvariant relations in a balanced tensor product. The detector applies to every power series \(f\); we restrict to the subspace \(U_m\) only because its elements can be lifted by the construction in Section~\(\ref{sec:lifts}\).

\Needspace{10\baselineskip}
\begin{lemma}[Power-series detector]\label{lem:detector}
The map
\[
 \theta_m:P\longrightarrow A_m^2(P)_\Gamma,
 \qquad f\longmapsto[f\star1],
\]
has a kernel of at most countable dimension. Its restriction to
\begin{equation}\label{eq:parity}
 U_m=\{f\in tP:f(-t)=(-1)^{m+1}f(t)\}
 =\begin{cases}
   t\Q[[t^2]],&m\text{ even},\\
   t^2\Q[[t^2]],&m\text{ odd}
  \end{cases}
\end{equation}
has image of dimension $\cont$.
\end{lemma}

\begin{proof}
Let $\F=\Q((t))$, let $\sigma(f)(t)=f(-t)$, and set
\[
 \K=\Frac\bigl(\Q[e^{qt}:q\in\Q]\bigr)\subseteq\F.
\]
The field $\K$ is countable and is preserved by $\sigma$. There is a well-defined
linear map
\begin{equation}\label{eq:tensor-detector}
 A_m^2(P)_\Gamma\longrightarrow\F\otimes_{\K}\F,
 \qquad
 [f\star g]\longmapsto
 f\otimes\sigma(g)+(-1)^m g\otimes\sigma(f).
\end{equation}
The symmetry relation is immediate, and the diagonal action is killed
because $e^{qt}\in\K$ and $\sigma(e^{qt})=e^{-qt}$, so~\eqref{eq:tensor-detector} is correctly defined.
The image of $\theta_m(f)$ is
\[
 f\otimes1+(-1)^m1\otimes\sigma(f).
\]
If \(f\notin\K\), then \(1\) and \(f\) are linearly independent over \(\K\). Define a \(\K\)-linear functional on their span by
\[
\ell(1)=0,\qquad \ell(f)=1,
\]and extend it to a \(\K\)-linear functional \(\ell:\F\to\K\). The induced map
\[
\ell\otimes_{\K}\mathrm{id}:
\F\otimes_{\K}\F\longrightarrow
\K\otimes_{\K}\F\cong\F
\]sends
\[
f\otimes1+(-1)^m1\otimes\sigma(f)
\]to
\[
\ell(f)\cdot1+(-1)^m\ell(1)\sigma(f)=1.
\]Hence this tensor is nonzero. Since it is the image of
\(\theta_m(f)\) under the detector map, \(\theta_m(f)\neq0\).
Therefore
\[
\ker\theta_m\subseteq\K.
\]

The space $U_m$ is isomorphic to $\Q[[t]]$, whose cardinality and
dimension over the countable field $\Q$ are both $\cont$.
Quotienting by a space of dimension at most $\aleph_0$ leaves dimension
$\cont$. This proves the second assertion.
\end{proof}

\section{Lifting the detected classes}
\label{sec:lifts}

We now construct cycles in the total complex
$\Tot B_{*,*}(M)$ of~\eqref{eq:total} whose images under the chain map
induced by $\Phi$ in Proposition~\ref{prop:lifting} represent the classes
detected in Lemma~\ref{lem:detector}.

\begin{proposition}[Lifting lemma]\label{prop:lifting}
Let $M$, $D$, and $L_D$ be as in Lemma~\ref{lem:split}, and let
$m\geq2$. Suppose there is a filtered dgl morphism
$\psi:M\to P[m-1]$ satisfying $\psi D=t\psi$, and a cycle
$b\in M_{m-1}$ with $\psi(b)=1$.
Then the split morphism
\[
 \Phi=\mathrm{id}\ltimes\psi:L_D\longrightarrow G_m
\]
induces, after Lie realization~\eqref{eq:Lie_realization}, a map on $H_{2m}$ whose image has
dimension $\cont$.
More precisely, its image contains $\theta_m(U_m)$ under
\eqref{eq:target-homology}.
\end{proposition}

\begin{proof}
Put $b_j=D^jb$. These are cycles and
\begin{equation}\label{eq:bracket-sign}
 [b_i,b_j]=(-1)^m[b_j,b_i].
\end{equation}
For $j\geq1$ with $j\equiv m+1\pmod2$, set
\begin{equation}\label{eq:telescope}
 u_j=\frac12\sum_{i=0}^{j-1}(-1)^i[b_{j-1-i},b_i].
\end{equation}
Since $D$ is a degree-zero derivation, the intermediate terms cancel:
\[
 2Du_j=[b_j,b]+(-1)^{j-1}[b,b_j]=2[b_j,b].
\]
For $f(t)=\sum_j\lambda_jt^j\in U_m$, define
\begin{equation}\label{eq:series-lifts}
 r_f=\sum_j\lambda_jb_j,\qquad
 u_f=\sum_j\lambda_ju_j,\qquad
 v_f=B(D)u_f,
 \qquad B(z)=\frac{z}{e^z-1}\in\Q[[z]].
\end{equation}
These expressions converge: $b_j,u_j\in F^{j+1}M$, and $D$ raises
filtration. They are cycles for the internal differential, and
\begin{equation}\label{eq:difference}
 Du_f=[r_f,b],\qquad (e^D-1)v_f=[r_f,b].
\end{equation}
The second equality is the formal identity $(e^z-1)B(z)=z$ applied
to $D$; it requires no inverse of $D$.

In the total complex \eqref{eq:total}, consider
\begin{equation}\label{eq:lifted-cycle}
 Q_f=(sr_f)(sb)+(-1)^m[1]\otimes sv_f.
\end{equation}
Its terms have bidegrees $(0,2m)$ and $(1,2m-1)$, respectively.
Equations \eqref{eq:CE} and \eqref{eq:difference} give
\[
 \partial_{\mathrm{tot}}Q_f
 =(-1)^{m-1}s[r_f,b]+(-1)^m s(e^D-1)v_f=0.
\]
Here $[1]$ is the bar simplex for $1\in(\Q,+)$, whose identity is $0$.

Under $\psi$, the elements $r_f$ and $b$ map to $f$ and $1$.
Every $u_j$ maps to zero, because $P[m-1]$ is abelian; hence
$\psi(v_f)=0$. Thus $Q_f$ maps to $(sf)(s1)$, representing
$\theta_m(f)$. The naturality in Lemma~\ref{lem:chains} proves the
claimed containment. By Lemma~\ref{lem:detector} the image has
dimension at least $\cont$, and it has dimension at most $\cont$
because $|A_m^2(P)_\Gamma|\leq\cont$.
\end{proof}

\section{The completion map when \texorpdfstring{$H_1(X)=0$}{H1(X)=0}}
\label{sec:main-proof}

It remains to connect the general lifting construction with the space
$S^1\vee X$. A minimal model makes the homology of $X$ visible in its
indecomposable generators, so a bottom-degree class supplies the cycle
required by Proposition~\ref{prop:lifting}. We must also track the
original space inside its completion: this will distinguish the newly
detected classes from those already in the image of $H_{2m}(X)$.
The wedge model of Lemma~\ref{lem:model} records both features.

\begin{proof}[Proof of Theorem~\ref{thm:main} for $m\geq2$]
Use the models $M_X=(\LL(V),d)$ and $L$ of Lemma~\ref{lem:model}.
Choose $0\neq b\in V_{m-1}$ and a homogeneous basis of $V$ containing
$b$. Since all generators have degree at least $m-1$ and the differential
is decomposable, $db=0$: a term of $db$ would have degree at least
$2m-2$, whereas $|db|=m-2$.

Let $N_X$ be the kernel of $L\to\Q a$ that kills $V$, with the
filtration induced from $L$. Then $N_X$ is positive and complete, and
\[
 L=\Q a\ltimes_DN_X,\qquad D=\ad_a,
 \qquad D(F^rN_X)\subseteq F^{r+1}N_X.
\]
Define
\begin{equation}\label{eq:quotient}
 \Phi:L\longrightarrow G_m,
 \qquad a\longmapsto a,\quad b\longmapsto1,\quad
 v\longmapsto0\quad(v\neq b\text{ in the chosen basis}).
\end{equation}
This is a filtered dgl morphism. Indeed, a bracket word of length $r$
with nonzero image, other than $a$ itself, contains one $b$ and $r-1$
copies of $a$, and maps into $t^{r-1}P$.
The map therefore extends to the completion. Every $dv$ maps to zero
because it is a sum of brackets in $V$ and $P[m-1]$ is abelian.
Its restriction $\psi:N_X\to P[m-1]$ satisfies $\psi D=t\psi$ and
$\psi(b)=1$. In fact $\Phi$ is surjective, since
$\sum_j\lambda_j\ad_a^j(b)$ maps to $\sum_j\lambda_jt^j$.

Let $h:\comp Y\to\real{G_m}$ be the realization of $\Phi$; here we are using
Lemma~\ref{lem:model}. By Proposition~\ref{prop:lifting},
\begin{equation}\label{eq:large-image}
 \dim_{\Q}\operatorname{im}H_{2m}(h)=\cont.
\end{equation}
We compare this with the image of $H_{2m}(h\eta_Y)$.
On the $X$ summand, the Lie morphism factors as
\[
 M_X\longrightarrow\Q c[m-1]
       \longrightarrow P[m-1]\longrightarrow G_m,
 \qquad b\longmapsto c\longmapsto1,
\]
where $\Q c[m-1]$ is abelian and every other basis vector of $V$ maps
to zero. By \eqref{eq:wedge-unit-square}, the restriction of
$h\eta_Y$ to $X$ therefore factors, up to homotopy, through $K(\Q,m)$.
Since $H_{2m}(Y)=H_{2m}(X)$, it follows that
\begin{equation}\label{eq:small-image}
 \operatorname{im}H_{2m}(h\eta_Y)
 \subseteq\Q[1\star1]\subseteq A_m^2(P)_\Gamma.
\end{equation}
This subspace has dimension at most one, and is zero when $m$ is odd.

The map $H_{2m}(h)$ induces a surjection
\[
 \coker H_{2m}(\eta_Y)\longrightarrow
 \frac{\operatorname{im}H_{2m}(h)}
      {\operatorname{im}H_{2m}(h\eta_Y)}.
\]
By \eqref{eq:large-image} and \eqref{eq:small-image}, the space on the
right has dimension $\cont$. This proves \eqref{eq:main}.
\end{proof}

The use of the bottom nonzero degree in this proof serves only to produce
a cycle with a nonzero indecomposable component. More generally, any
nonzero rational homology class in the image of the Hurewicz homomorphism
supplies such a cycle and gives the same obstruction in its doubled
degree.

\begin{proof}[Proof of the final assertion of Theorem~\ref{thm:main}]
Assume that $X$ is simply connected and that its rational Hurewicz map
is nonzero in degree $n\geq2$.
In a minimal Lie model $M_X=\LL(V)$, the rational Hurewicz map is
induced by projection of cycles to the indecomposables $V$
(see \cite{Quillen}). Choose a cycle $b\in(M_X)_{n-1}$ whose projection is
nonzero, and a functional $\ell:V_{n-1}\to\Q$ taking that projection
to $1$. Define the quotient \eqref{eq:quotient} using $\ell$ in degree
$n-1$ and zero in the other degrees. The lifting lemma applies to
this cycle $b$, and the restriction to $X$ still factors through
$K(\Q,n)$. The preceding cokernel argument is unchanged and
gives~\eqref{eq:spherical-main}.
\end{proof}

\section{Nonzero first homology and examples}
\label{sec:examples}

When $H_1(X)\neq0$, the wedge model has additional degree-zero
generators, so its positive-ideal description no longer applies.
In this case a nonzero first cohomology class and a loop on which it
evaluates nontrivially give a reduction to two circles. We first recall
the precise cardinal bound furnished by the construction of Ivanov
and Mikhailov.

\begin{lemma}[Two-circle cardinality bound]\label{lem:two-circles}
For $W=S^1\vee S^1$, one has
\[
 \dim_{\Q}H_2(\comp W)=\cont.
\]
\end{lemma}

\begin{proof}
Ivanov and Mikhailov prove that this homology is uncountable
\cite[Theorem~1]{IM}. Their construction in Section~5 gives the sharper
dimension bound stated here. More explicitly, for
$q=(q_i)_{i\geq1}\in\{0,1\}^{\mathbb N}$, the construction supplies series
\[
 f_q=\sum_{i\geq3}n_i(q)x^{i-1},\qquad n_{2i+1}(q)=q_i.
\]
The displayed coefficients distinguish the binary sequences, so this
family has cardinality $\cont$. The detector $\Theta_{\Q}$ has
countable kernel by Propositions 2.1 and 3.1 of \cite{IM}, hence its
restriction to this family has countable fibers. Its image consequently
has cardinality $\cont$. Proposition 4.1 and Section 5 of \cite{IM}
give lifts of all these classes through $H_2(\comp W)$; see especially
pp.~1248--1249. Over the countable field $\Q$, the span of a set of
cardinality $\cont$ has dimension $\cont$.
For the reverse bound, $\comp W\simeq K(\widehat F_{2,\Q},1)$, where
$\widehat F_{2,\Q}$ is the rational pronilpotent completion of the
free group on two generators. It is an inverse limit of countable
groups, so its bar complex has cardinality at most $\cont$ in each
degree.
\end{proof}

\begin{proof}[Proof of Theorem~\ref{thm:main} for $m=1$]
Choose $u:X\to K(\Q,1)$ representing a nonzero element of $H^1(X;\Q)$.
There is a based loop $\gamma:S^1\to X$ such that $u\gamma$ induces a
nonzero map on $H_1$, hence a rational homology equivalence.
Put $Z=S^1\vee K(\Q,1)$ and consider
\[
 W=S^1\vee S^1
   \xrightarrow{\ i=\mathrm{id}\vee\gamma\ }Y
   \xrightarrow{\ p=\mathrm{id}\vee u\ }Z.
\]
The composite $pi$ is a rational homology equivalence. Its completion
is therefore a weak equivalence, so
$H_2(\comp p):H_2(\comp Y)\to H_2(\comp Z)$ is surjective.
On the other hand, $H_2(Z)=0$, and naturality gives
$H_2(\comp p)H_2(\eta_Y)=0$.
Consequently $\coker H_2(\eta_Y)$ surjects onto
$H_2(\comp Z)\cong H_2(\comp W)$, which has dimension $\cont$ by
Lemma~\ref{lem:two-circles}.
\end{proof}

The theorem gives a lower bound without restricting the size of the
rational homology of $X$. If that homology has at most countable total
dimension, a cardinality estimate on its complete Lie model gives a
matching upper bound.

\begin{corollary}\label{cor:exact-size}
If, in addition to the hypotheses of Theorem~\ref{thm:main},
$\widetilde H_*(X)$ has at most countable total dimension, then
\[
 \dim_{\Q}H_{2m}(\comp Y)
 =\dim_{\Q}\coker H_{2m}(\eta_Y)=\cont.
\]
\end{corollary}

\begin{proof}
A minimal complete Lie model of $Y$ has countably many generators.
For each fixed bracket length, the corresponding vector space is
countable as a set, so the completion has cardinality at most $\cont$. A simplex in its realization is a
filtered morphism from a free complete Lie algebra on finitely many
generators and is determined by their images. There are therefore at
most $\cont$ simplices in each degree, giving the upper bound on
homology. Theorem~\ref{thm:main} supplies the lower bound.
\end{proof}

\begin{example}[Spheres]
For every $r\geq1$, the space $S^1\vee S^r$ is $\Q$-bad and
\[
 \dim_{\Q}H_{2r}\bigl(\comp{(S^1\vee S^r)}\bigr)=\cont.
\]
For $r=2$, the first telescoping element is $u_1=\tfrac12[b,b]$;
for $r=3$, it is $u_2=[Db,b]$. These are the two parity cases of
\eqref{eq:telescope}.
\end{example}

\begin{example}[Eilenberg--Mac Lane spaces]
Let $A$ be an abelian group with $A\otimes\Q\neq0$, and let $r\geq1$.
The first nonzero rational homology of $K(A,r)$ occurs in degree $r$.
Theorem~\ref{thm:main} shows that $S^1\vee K(A,r)$ is $\Q$-bad and
that its completion map has a cokernel of dimension at least $\cont$
on $H_{2r}$, regardless of the cardinality of $A$.
If $A\otimes\Q$ has at most countable dimension, the rational homology
of $K(A,r)$, the free graded-commutative algebra on $A\otimes\Q$ in
degree $r$, has at most countable total dimension. In this case
Corollary~\ref{cor:exact-size} gives equality with $\cont$.
\end{example}

\section{Replacing the circle by a space with nonzero first homology}
\label{sec:generalisation}

The passage from $S^1\vee X$ to $A\vee X$ uses only naturality and the
fact that rational completion carries rational homology equivalences
to weak equivalences \cite[Chapter~I, \S5]{BK}; see also
\cite[Introduction]{Completion}. We keep track of the cokernel itself,
so that the argument applies even when the homology of $A$ or $X$ has
arbitrary dimension.

For a connected based space $T$ and $j\geq1$, write
\begin{equation}\label{eq:completion-cokernel}
 \mathcal D_j(T)=\coker\bigl(H_j(T)
             \xrightarrow{\,H_j(\eta_T)\,}H_j(\comp T)\bigr).
\end{equation}
For a based map $f:T\to T'$, naturality gives a commutative square
\[
 \begin{tikzcd}[column sep=large]
 H_j(T) \arrow[r,"H_j(f)"] \arrow[d,"H_j(\eta_T)"'] &
 H_j(T') \arrow[d,"H_j(\eta_{T'})"] \\
 H_j(\comp T) \arrow[r,"H_j(\comp f)"'] & H_j(\comp{T'}).
 \end{tikzcd}
\]
It therefore induces a linear map
$\mathcal D_j(f):\mathcal D_j(T)\to\mathcal D_j(T')$, functorially in
$f$. If $f$ is a rational homology equivalence, both horizontal maps
are isomorphisms, and hence $\mathcal D_j(f)$ is an isomorphism.

\begin{proposition}[Transfer from a circle summand]
\label{prop:circle-transfer}
Let $A$ be a connected based CW complex with $H_1(A)\neq0$.
There is a based loop $\gamma:S^1\to A$ such that, for every connected
based CW complex $X$ and every $j\geq1$, the map
\[
 \gamma\vee\mathrm{id}_X:S^1\vee X\longrightarrow A\vee X
\]
induces a split injection
\begin{equation}\label{eq:split-cokernel}
 \mathcal D_j(S^1\vee X)\lhook\joinrel\longrightarrow
 \mathcal D_j(A\vee X).
\end{equation}
\end{proposition}

\begin{proof}
Choose a nonzero linear functional $\lambda:H_1(A)\to\Q$.
The homomorphism
\[
 \pi_1(A)\longrightarrow H_1(A;\Z)
       \longrightarrow H_1(A)\xrightarrow{\lambda}\Q
\]
is nonzero: the image of integral first homology spans $H_1(A)$ over
$\Q$. Choose a based loop $\gamma$ on which it takes a nonzero value,
and rescale $\lambda$ so that this value is $1$.
This homomorphism is represented by a based map $u:A\to K(\Q,1)$.
Equivalently, $u$ represents the cohomology class corresponding to
$\lambda$ under the universal coefficient isomorphism
\[
 H^1(A;\Q)\cong\operatorname{Hom}_{\Q}(H_1(A),\Q).
\]
Thus $u\gamma:S^1\to K(\Q,1)$ induces an isomorphism on $H_1$.
The rational homology of $K(\Q,1)$ is $\Q$ in degrees zero and one
and zero in higher degrees. Indeed, the group $(\Q,+)$ is the filtered
union of its infinite cyclic subgroups, and group homology commutes
with filtered colimits. It follows that $u\gamma$ is a rational
homology equivalence.

Put $W=S^1\vee X$, $Z=A\vee X$, and $V=K(\Q,1)\vee X$, and consider
\begin{equation}\label{eq:circle-transfer-maps}
 W\xrightarrow{\ i=\gamma\vee\mathrm{id}_X\ }Z
  \xrightarrow{\ p=u\vee\mathrm{id}_X\ }V.
\end{equation}
The composite $h=pi=(u\gamma)\vee\mathrm{id}_X$ is a rational
homology equivalence, by the direct-sum decomposition of reduced
homology of a wedge. Consequently $\mathcal D_j(h)$ is an
isomorphism. Functoriality gives
\[
 \mathcal D_j(p)\mathcal D_j(i)=\mathcal D_j(h).
\]
The map
\begin{equation}\label{eq:cokernel-retraction}
 r_j=\mathcal D_j(h)^{-1}\mathcal D_j(p):
 \mathcal D_j(Z)\longrightarrow\mathcal D_j(W)
\end{equation}
satisfies $r_j\mathcal D_j(i)=\mathrm{id}$. This proves the split
injection~\eqref{eq:split-cokernel}.
\end{proof}

\begin{proof}[Proof of Corollary~\ref{cor:generalisation}]
Apply Proposition~\ref{prop:circle-transfer} with $j=2m$.
Theorem~\ref{thm:main} gives
\[
 \dim_{\Q}\mathcal D_{2m}(S^1\vee X)\geq\cont.
\]
Since this vector space is a direct summand of
$\mathcal D_{2m}(A\vee X)$, inequality~\eqref{eq:generalised-main}
follows. The completion has homology dimension at least that of this
cokernel, and the nonzero cokernel implies $\Q$-badness.
If the rational Hurewicz map of $X$ is nonzero in degree $n\geq2$,
use the same split injection with $j=2n$ and the final assertion of
Theorem~\ref{thm:main}. This gives~\eqref{eq:generalised-spherical}.
\end{proof}

\begin{corollary}[Goodness criterion for a general summand]
\label{cor:generalised-criterion}
Let $A$ and $X$ be connected based CW complexes with $H_1(A)\neq0$.
Then $A\vee X$ is $\Q$-good if and only if $A$ is $\Q$-good and
$\widetilde H_*(X)=0$.
\end{corollary}

\begin{proof}
If $A\vee X$ is $\Q$-good, Corollary~\ref{cor:generalisation} forces
$\widetilde H_*(X)=0$. Whenever this vanishing holds, the collapse
$c:A\vee X\to A$ is a rational homology equivalence. Naturality gives
\[
 H_*(\comp c)H_*(\eta_{A\vee X})=H_*(\eta_A)H_*(c),
\]
where $H_*(c)$ and $H_*(\comp c)$ are isomorphisms. Thus
$H_*(\eta_{A\vee X})$ is an isomorphism if and only if
$H_*(\eta_A)$ is, proving both directions.
\end{proof}

\appendix

\section{Auxiliary comparison results}
\label{app:comparison-results}

\begin{proof}[Proof of Lemma~\ref{lem:mc-comparison}]
For \(r\geq2\), put \(M_r=M/F^rM\). Since
\[
 [F^pM,F^qM]\subseteq F^{p+q}M,
\]
each \(M_r\) is nilpotent.

On \(M_r\), the induced filtration and the lower central filtration give canonically isomorphic realizations and Maurer--Cartan simplicial sets, naturally with respect to the quotient maps. The standard nilpotent comparison
provides natural simplicial maps
\[
 \iota_r:\real{M_r}\longrightarrow\MC_\bullet(M_r),
 \qquad
 \rho_r:\MC_\bullet(M_r)\longrightarrow\real{M_r},
\]
with
\[
 \rho_r\iota_r=\operatorname{id}_{\real{M_r}},
\]
and both maps are weak equivalences. Moreover, there is a natural
simplicial identification
\[
 \real{M_r}\cong\gamma_\bullet(M_r)
\]
with the gauge-fixed nerve. These assertions follow from the
natural comparison construction in
\cite[Theorems~0.1--0.2 and Section~3]{InfinityQuillen};
see also \cite[Chapter~11]{LieModels}.

The quotient maps \(M_{r+1}\twoheadrightarrow M_r\) are surjective
morphisms of nilpotent dgls. Hence the transition maps
\[
 \MC_\bullet(M_{r+1})\longrightarrow\MC_\bullet(M_r)
\]
are Kan fibrations by \cite[Proposition~4.7]{Getzler}.
Likewise, \cite[Theorem~5.10]{Getzler} implies that
\[
 \gamma_\bullet(M_{r+1})\longrightarrow\gamma_\bullet(M_r)
\]
are Kan fibrations. Under the natural identification above, these
are the transition maps of the realization tower. Both towers
therefore consist of Kan complexes with Kan fibration transition
maps.

Naturality makes the maps \(\iota_r\) and \(\rho_r\) compatible
with the transition maps. Completeness and the definition of the
completed tensor product give natural identifications
\[
 \real M\cong\varprojlim_r\real{M_r},
 \qquad
 \MC_\bullet(M)\cong\varprojlim_r\MC_\bullet(M_r).
\]
Define
\[
 \iota_M=\varprojlim_r\iota_r,
 \qquad
 \rho_M=\varprojlim_r\rho_r.
\]
The ordinary inverse limit of a tower of Kan complexes with Kan
fibration transition maps computes its homotopy inverse limit.
Consequently, the levelwise weak equivalences \(\iota_r\) induce
a weak equivalence \(\iota_M\).

Passing to inverse limits in
\(\rho_r\iota_r=\operatorname{id}_{\real{M_r}}\) yields
\[
 \rho_M\iota_M=\operatorname{id}_{\real M}.
\]
Thus \(\rho_M\) is also a weak equivalence by the
two-out-of-three property. Finally, naturality at each quotient
passes to the inverse limit, proving the asserted naturality
and compatibility with automorphisms.
\end{proof}

\begin{lemma}[Exponentiating a filtration-raising derivation]
\label{lem:exponential-action}
Let \(M\) be a complete filtered dgl and let \(D\) be a degree-zero
chain derivation satisfying
\[
 D(F^rM)\subseteq F^{r+1}M
 \qquad (r\geq1).
\]
Then, for every \(q\in\Q\), the exponential
\[
 T_q=e^{qD}
\]
defines a filtration-preserving dgl automorphism of \(M\). These
automorphisms satisfy
\[
 T_qT_r=T_{q+r},
 \qquad
 T_q^{-1}=T_{-q},
\]
and hence define an action of \(\Gamma=(\Q,+)\) on \(M\).

Moreover, if \(\gamma\) denotes the lower central filtration of \(M\),
then
\[
 T_q(\gamma^rM)=\gamma^rM
 \qquad (r\geq1).
\]
Thus \(T_q\) also defines a filtered automorphism of \((M,\gamma)\).
In particular, the identity morphism
\[
 \iota:(M,\gamma)\longrightarrow(M,F)
\]
is \(\Gamma\)-equivariant.
\end{lemma}

\begin{proof}
For every \(N\geq2\), the derivation \(D\) induces a degree-zero chain
derivation
\[
 D_N:M/F^NM\longrightarrow M/F^NM.
\]
Indeed, \(D(F^NM)\subseteq F^{N+1}M\subseteq F^NM\). Iterating the
filtration-raising condition gives
\[
 D^j(F^rM)\subseteq F^{r+j}M.
\]
In particular,
\[
 D^{N-1}(M)\subseteq F^NM,
\]
so \(D_N^{N-1}=0\).

Consequently, for \(q\in\Q\), the exponential
\[
 T_{q,N}=e^{qD_N}
 =\sum_{j=0}^{N-1}\frac{q^jD_N^j}{j!}
\]
is a finite sum. The derivation identity implies, by induction on
\(n\), that
\[
 D_N^n[x,y]
 =
 \sum_{i=0}^n
 \binom{n}{i}
 [D_N^ix,D_N^{n-i}y].
\]
It follows that
\begin{align*}
 e^{qD_N}[x,y]
 &=
 \sum_{n\geq0}\frac{q^n}{n!}D_N^n[x,y] \\
 &=
 \sum_{i,j\geq0}
 \frac{q^{i+j}}{i!\,j!}
 [D_N^ix,D_N^jy] \\
 &=
 [e^{qD_N}x,e^{qD_N}y].
\end{align*}
All sums here are finite because \(D_N\) is nilpotent. Since \(D_N\)
commutes with the differential, \(T_{q,N}\) is a dgl morphism.

The usual exponential identities for commuting nilpotent operators give
\[
 T_{q,N}T_{r,N}=T_{q+r,N},
 \qquad
 T_{q,N}T_{-q,N}=\operatorname{id}.
\]
Thus \(T_{q,N}\) is a dgl automorphism with inverse \(T_{-q,N}\).

The automorphisms \(T_{q,N}\) are compatible with the quotient maps
\[
 M/F^{N+1}M\longrightarrow M/F^NM.
\]
Using completeness,
\[
 M\cong\varprojlim_N M/F^NM,
\]
we therefore obtain a dgl automorphism
\[
 T_q=\varprojlim_NT_{q,N}
     =e^{qD}:M\longrightarrow M.
\]
The quotientwise identities pass to the inverse limit and give
\[
 T_qT_r=T_{q+r},
 \qquad
 T_q^{-1}=T_{-q}.
\]
Moreover, if \(x\in F^rM\), then \(D^jx\in F^{r+j}M\subseteq F^rM\);
hence \(T_q(F^rM)=F^rM\). Thus \(T_q\) is filtration-preserving for
the specified filtration \(F\).

Finally, every Lie algebra automorphism preserves the lower central
series. Indeed, writing
\[
 \gamma^1M=M,
 \qquad
 \gamma^{r+1}M=[M,\gamma^rM],
\]
an induction gives
\[
 \begin{aligned}
 T_q(\gamma^{r+1}M)
 &=
 T_q([M,\gamma^rM])  \\
 &=
 [T_q(M),T_q(\gamma^rM)] \\
 &=
 [M,\gamma^rM]
 =\gamma^{r+1}M.
 \end{aligned}
\]
Therefore \(T_q\) is also a filtered automorphism of \((M,\gamma)\).
Since \(\iota\) is the identity on the underlying dgl, one has
\[
 \iota T_q=T_q\iota,
\]
so \(\iota:(M,\gamma)\to(M,F)\) is \(\Gamma\)-equivariant.
\end{proof}

\begin{lemma}[Equivariant bar comparison]
\label{lem:equivariant-bar-comparison}
Let \(\Gamma\) be a discrete group.

\begin{enumerate}
\item
If \(f:A_*\to A'_*\) is a \(\Gamma\)-equivariant
quasi-isomorphism of nonnegative chain complexes of \(\Q\)-vector
spaces, then
\[
 \Tot C_*(\Gamma;f):
 \Tot C_*(\Gamma;A_*)
 \longrightarrow
 \Tot C_*(\Gamma;A'_* )
\]
is a quasi-isomorphism.

\item
For every simplicial \(\Gamma\)-set \(X\), there is a natural
quasi-isomorphism
\[
 C_*(E\Gamma\times_\Gamma X;\Q)
 \xrightarrow{\simeq}
 \Tot C_*\bigl(\Gamma;C_*(X;\Q)\bigr).
\]
Under this comparison, the topological fiber inclusion
\[
 i_X:X\longrightarrow E\Gamma\times_\Gamma X
\]
corresponds to the inclusion of the coefficient complex in bar degree
zero,
\[
 j_X:C_*(X;\Q)=C_0\bigl(\Gamma;C_*(X;\Q)\bigr)
 \longrightarrow
 \Tot C_*\bigl(\Gamma;C_*(X;\Q)\bigr).
\]
\end{enumerate}
Both assertions are natural in equivariant maps.
\end{lemma}

\begin{proof}
For fixed bar degree \(p\), the normalized group-bar complex is a
direct sum of copies of its coefficient complex:
\[
 C_p(\Gamma;A_*)
 \cong
 \bigoplus_{(g_1,\ldots,g_p)\in(\Gamma\setminus\{e\})^p}A_*.
\]
Consequently, if \(f:A_*\to A'_*\) is a quasi-isomorphism, then
\(C_p(\Gamma;f)\) is a quasi-isomorphism for every \(p\). Filtering the
total complex by bar degree gives a first-quadrant spectral sequence
with
\[
 E^1_{p,q}(A)
 =
 H_q\bigl(C_p(\Gamma;A_*)\bigr)
 \cong
 C_p\bigl(\Gamma;H_q(A)\bigr).
\]
The map induced by \(f\) is an isomorphism on the \(E^1\)-page, so the
spectral-sequence comparison theorem proves the first assertion.

For the second assertion, consider the bisimplicial Borel construction
\[
 \mathcal B_{p,q}=\Gamma^p\times X_q,
\]
with the usual inhomogeneous bar operators in the first direction and
the simplicial operators of \(X\) in the second. Its diagonal is
\[
 \operatorname{diag}\mathcal B
 \cong E\Gamma\times_\Gamma X.
\]
The normalized Eilenberg--Zilber comparison for bisimplicial
\(\Q\)-vector spaces gives the natural quasi-isomorphism
\[
 C_*(E\Gamma\times_\Gamma X;\Q)
 \xrightarrow{\simeq}
 \Tot C_*\bigl(\Gamma;C_*(X;\Q)\bigr).
\]

In the inhomogeneous model, the fiber inclusion is given in degree
\(n\) by
\[
 x\longmapsto(x;e,\ldots,e).
\]
All group coordinates are degenerate, so after horizontal
normalization only bar degree zero remains. The Eilenberg--Zilber
comparison therefore identifies this map with
\[
 j_X:C_*(X;\Q)\longrightarrow
 \Tot C_*\bigl(\Gamma;C_*(X;\Q)\bigr).
\]
Naturality follows from the naturality of the bar construction and of
the normalized Eilenberg--Zilber comparison.
\end{proof}


\begin{thebibliography}{99}
\raggedright

\bibitem{BousfieldLocalization}
A.~K. Bousfield,
\emph{The localization of spaces with respect to homology},
Topology \textbf{14} (1975), 133--150.
\href{https://doi.org/10.1016/0040-9383(75)90023-3}
{doi:10.1016/0040-9383(75)90023-3}.

\bibitem{BousfieldTowers}
A.~K. Bousfield,
\emph{Homological localization towers for groups and $\pi$-modules},
Mem. Amer. Math. Soc. \textbf{10} (1977), no.~186.

\bibitem{BousfieldPadic}
A.~K. Bousfield,
\emph{On the $p$-adic completions of nonnilpotent spaces},
Trans. Amer. Math. Soc. \textbf{331} (1992), 335--359.
\href{https://doi.org/10.1090/S0002-9947-1992-1038019-9}
{doi:10.1090/S0002-9947-1992-1038019-9}.

\bibitem{BK}
A.~K. Bousfield and D.~M. Kan,
\emph{Homotopy Limits, Completions and Localizations},
Lecture Notes in Mathematics \textbf{304}, Springer, 1972.

\bibitem{LieModels}
U.~Buijs, Y.~F\'elix, A.~Murillo and D.~Tanr\'e,
\emph{Lie Models in Topology},
Progress in Mathematics \textbf{335}, Birkh\"auser, 2020.
\href{https://doi.org/10.1007/978-3-030-54430-0}
{doi:10.1007/978-3-030-54430-0}.

\bibitem{InfinityQuillen}
U.~Buijs, Y.~F\'elix, A.~Murillo and D.~Tanr\'e,
\emph{The infinity Quillen functor, Maurer--Cartan elements and DGL
realizations}, arXiv:1702.04397v2, 2018.
\url{https://arxiv.org/abs/1702.04397}.

\bibitem{CodensityCompletion}
E.~Dror Farjoun and S.~O. Ivanov,
\emph{Bousfield--Kan completion as a codensity $\infty$-monad},
arXiv:2507.08414v3, 2025.
\url{https://arxiv.org/abs/2507.08414v3}.

\bibitem{Completion}
Y.~F\'elix, M.~Fuentes and A.~Murillo,
\emph{A Lie characterization of the Bousfield--Kan $\Q$-completion and
$\Q$-good spaces}, Trans. Amer. Math. Soc., to appear.
\href{https://doi.org/10.1090/tran/9855}{doi:10.1090/tran/9855};
arXiv:2407.02812v2, 2026.

\bibitem{Realizations}
Y.~F\'elix, M.~Fuentes and A.~Murillo,
\emph{All known realizations of complete Lie algebras coincide},
Algebr. Geom. Topol. \textbf{25} (2025), 1155--1167.
\href{https://doi.org/10.2140/agt.2025.25.1155}
{doi:10.2140/agt.2025.25.1155}.

\bibitem{Getzler}
E.~Getzler,
\emph{Lie theory for nilpotent $L_\infty$-algebras},
Ann. of Math. (2) \textbf{170} (2009), 271--301.
\href{https://doi.org/10.4007/annals.2009.170.271}
{doi:10.4007/annals.2009.170.271}.

\bibitem{IM}
S.~O. Ivanov and R.~Mikhailov,
\emph{A finite $\Q$-bad space},
Geom. Topol. \textbf{23} (2019), 1237--1249.
\href{https://doi.org/10.2140/gt.2019.23.1237}
{doi:10.2140/gt.2019.23.1237}.

\bibitem{Quillen}
D.~Quillen,
\emph{Rational homotopy theory},
Ann. of Math. (2) \textbf{90} (1969), 205--295.
\href{https://doi.org/10.2307/1970725}{doi:10.2307/1970725}.

\end{thebibliography}
\end{document}